\documentclass{scrartcl}

\usepackage[backref=page]{hyperref}
\hypersetup{
pdftitle={Conjugation, loop and closure invariants of the iterated-integrals signature},
pdfsubject={Given a feature set for the shape of a closed loop, it is natural to ask which features in that set do not change when the starting point of the path is moved. For example, in two dimensions, the area enclosed by the path does not depend on the starting point. In the present article, we characterize such loop invariants among all those features known as interated integrals of a given path. Furthermore, we relate these to conjugation invariants, which are a canonical object of study when treating (tree reduced) paths as a group with multiplication given by the concatenation. Finally, closure invariants are a third class in this context which is of particular relevance when studying piecewise linear trajectories, e.g. given by linear interpolation of time series.},
pdfauthor={Joscha Diehl, Rosa Preiß, Jeremy Reizenstein},
pdfkeywords={invariant features; concatenation of paths; combinatorial necklaces; shuffle algebra; free Lie algebra; signed area; signed volume; tree-like equivalence}
}
\usepackage{amsmath,amsthm,amsfonts}
\usepackage{cleveref}
\usepackage{tikz}
\usepackage{shuffle}
\usepackage[section]{placeins}

\usepackage[linecolor=white,backgroundcolor=white,bordercolor=white,textsize=tiny]{todonotes}
\let\todon\todo
\renewcommand{\todo}[1]{\todon{\color{magenta}{#1}}}

\newtheorem{theorem}{Theorem}[section]
\theoremstyle{plain}

\newtheorem{conjecture}[theorem]{Conjecture}
\newtheorem{corollary}[theorem]{Corollary}

\newtheorem{example}[theorem]{Example}

\newtheorem{lemma}[theorem]{Lemma}

\newtheorem{proposition}[theorem]{Proposition}
\newtheorem{remark}[theorem]{Remark}

\theoremstyle{definition}
\newtheorem{definition}[theorem]{Definition}

\numberwithin{figure}{section}
\numberwithin{table}{section}

\newcommand{\R}{\mathbb{R}}
\newcommand{\word}[1]{{\color{cyan}\mathbf{\mathtt{#1}}}}
\newcommand{\w}[1]{\word{#1}}
\DeclareMathOperator{\sig}{\mathsf{sig}}
\DeclareMathOperator{\rot}{\mathsf{rot}}
\DeclareMathOperator{\rep}{\mathsf{rep}}
\DeclareMathOperator{\shift}{\mathsf{shift}}

\newcommand\e{\mathsf{e}}
\newcommand\DEF[1]{\textbf{#1}}
\newcommand\itensor{\mathbin{\bullet}}

\newcommand\rcl{\operatorname{{rcl}}}
\newcommand\lcl{\operatorname{{lcl}}}
\newcommand\im{\operatorname{im}}
\newcommand\id{\operatorname{id}}
\newcommand\N{\mathbb{N}}
\newcommand\areaconj{\mathsf{AreaConj}}
\newcommand\vol{\mathsf{vol}}
\newcommand\backwards[1]{{#1}^{-1}}

\usepackage[normalem]{ulem}

 \usepackage{xspace}

\newcommand\FL{\mathsf{FL}}
\newcommand{\oeis}[1]{\href{https://oeis.org/#1}{\nolinkurl{#1}}}

\newif\ifshow
\showtrue
\showfalse

\begin{document}
\author{Joscha Diehl (Greifswald University), Rosa Preiß (TU Berlin),\\ Jeremy Reizenstein (Meta AI)}
\title{Conjugation, loop and closure invariants of the iterated-integrals signature}
\date{December 27, 2024}
\maketitle

\begin{abstract}
Given a feature set for the shape of a closed loop, it is natural to ask which features in that set do not change when the starting point of the path is moved.
    For example, in two dimensions, the area enclosed by the path does not depend on the starting point.
    In the present article, we characterize such loop invariants among all those features known as interated integrals of a given path.
    Furthermore, we relate these to conjugation invariants,
    which are a canonical object of study when treating (tree reduced) paths as a group with multiplication given by the concatenation.
    Finally, closure invariants are a third class in this context which is of particular relevance when studying piecewise linear trajectories, e.g. given by linear interpolation of time series.\\[2ex]
\end{abstract}

\tableofcontents

\section{Introduction}

Finite-dimensional curves are a fundamental
object of study in both theoretical and applied mathematics.
Invariants of such curves are relevant both from
theoretical and practical perspectives.

Invariants with respect to group actions on the ambient space,
which are often relevant in data science \cite{morales2017physical},
have been studied through the lens of iterated-integrals in
\cite{Chen58_faithful,bib:DR2018,diehl2023moving}.

When acting on the curve itself,
the most well-known invariance built into the signature
is the reparametrization invariance \cite[Proposition 7.10]{friz2010multidimensional};
see also \cite{diehl2020time} for
related invariance to ``time-warping'' in the case
of discrete time series.

In the current work we consider invariance under certain
operations on the curves,
namely 
\begin{itemize}
  \item
    conjugation of paths
    (this corresponds to conjugation of group elements)\newline
    $\leadsto$ \DEF{conjugation invariants}
    (\Cref{sec:conjugation_invariants}),

  \item
    cyclic ``permutation'' of loops\newline
    $\leadsto$ \DEF{loop invariants} (\Cref{sec:loop_invariants}), and

  \item
    closure of paths\newline
    $\leadsto$ \DEF{closure invariants} (\Cref{sec:closure_invariants}).
\end{itemize}

\bigskip

\textbf{Notation}

\begin{itemize}

    \item 
    The \DEF{tensor algebra} over $\R^d$,
    \begin{align*}
        T(\R^d)
        := \bigoplus_{n\ge 0} T_n(\R^d)
        := \bigoplus_{n\ge 0} (\R^d)^{\otimes n},
    \end{align*}
    and the product space
    \begin{align*}
        T((\R^d)) := \prod_{n\ge 0} T_n(\R^d).
    \end{align*}
    Both come equipped with the
    \DEF{concatenation product} $\bullet$ (i.e. the tensor product),
    which we sometimes simply write as  $xy = x\bullet y$,
    as well as the \DEF{shuffle product} $\shuffle$.
    See for example \cite{reutenauer1993free}.
    The coproduct dual to the concatenation product
    is denoted $\Delta_\bullet$, the \DEF{deconcatenation}.

    On $T((\R^d))$
    we use the Lie bracket $[x,y] := xy - yx$.

    For example
    \begin{align*}
        \Delta_\bullet
        \word{143}
        =
        \e \otimes \word{143}
        +
        \word{1} \otimes \word{43}
        +
        \word{14} \otimes \word{3}
        +
        \word{143} \otimes \e.
    \end{align*}
where $\e$ denotes the empty word.
    
    Write $\pi_n$ for the projection onto $T_n := T_n(\R^d)$,
    on either space.
    Write $T_{\le n}(\R^d) := \bigoplus_{0\le i\le n} T_i(\R^d)$.
    
    We use the following bilinear pairing of $T((\R^d))$ and $T(\R^d)$
    \begin{align*}
       \langle \sum_w c_w w, v \rangle :=  c_v.
    \end{align*}

    An element $g \in T((\R^d))$ is \DEF{grouplike}
    if for all $\phi,\psi \in T(\R^d)$
    \begin{align*}
       \langle g, \phi \rangle
       \langle g, \psi \rangle
       =
       \langle g, \phi\shuffle\psi \rangle.
    \end{align*}

    Endowed with the commutator of the concatenation product,
    $[x,y] := xy - yx$, $T(\R^d)$ is a Lie algebra.
    The sub-Lie algebra generated by $\R^d$
    is denote $\FL(\R^d)$. Its elements are called \DEF{Lie polynomials}.
    It can be shown to be the \DEF{free Lie algebra over $\R^d$}.
    It is graded $\FL(\R^d) = \bigoplus_{n\ge 1} \FL_n(\R^d)$
    and its formal completion $\prod_{n\ge 1} \FL_n(\R^d)$
    is called the space of \DEF{Lie series}.

    \item
    Concatenation of curves $A,B$ with matching end- and startpoints
    is denoted $A \sqcup B$.

    \item
    $\N = \{0,1,2,\dots\}$ the non-negative integers.
    
\end{itemize}

\subsection{Acknowledgements}
For all \nolinkurl{AXXXXXX}-denoted integer sequences, we refer to the OEIS \cite{oeis}.
R.P.\ acknowledges support from DFG CRC/TRR 388 “Rough Analysis, Stochastic Dynamics and Related Fields”, Project A04.
R.P.\ would like to thank Kurusch Ebrahimi-Fard, Xi Geng and Michael Joswig for interesting suggestions.
Furthermore, R.P.\ thanks Felix Lotter for the collaboration on the project \cite{lotterpreiss24},
which gave valuable impulses for the present article, as well as acknowledges a guest status at MPI MiS Leipzig that facilitated the collaboration.
J.D.\ acknowledges support from DFG (Project 539875438, SPP ``Combinatorial Synergies'').

\section{Conjugation invariants}
\label{sec:conjugation_invariants}

\newcommand\EL{\phi}

We consider the groupoid of smooth paths
in $\R^d$ modulo thin homotopy (alternatively: up to reparametrization
and tree-like equivalence,
\cite{boedihardjo2016signature}).
Paths $A,B$ with fitting end- and start-point
can be concatenated, denoted by $A \sqcup B$,
and every path has an inverse $A^{-1}$
which is (the equivalence class of) the
path run backwards.
The iterated-integrals signature $\sig$ is a well-defined function
from this groupoid to $T((\R^d))$. For a path $B:[a,b]\to\R^n$, it is defined recursively on its components by $\left<\sig(B),\e\right>=1$ and, for any letter $i$ and any word $w$, $\left<\sig(B),w i\right>=\int_{t=a}^{t=b}\left<\sig(B|_{[a,t]}),w\right>dB^i(t)$. An introduction to signatures is given by \cite{OxSigIntro}. It is from this signature object that we seek to find invariants.

\begin{definition}
  An element $\EL\in T(\R^d)$ is called a \DEF{conjugation invariant} if for any two paths $A$ and $B$, $\left<\sig(A),\EL\right> = \left<\sig(B\sqcup A\sqcup B^{-1}),\EL\right>$.
Equivalently, if for any two paths $A$ and $B$, $\left<\sig(A\sqcup B),\EL\right> = \left<\sig(B\sqcup A), \EL\right>$.
\end{definition}

For example, single letters are conjugation invariant, since conjugate paths have the same total displacement.
Alternatively, this follows because addition of the displacement vectors is commutative.
The shuffle of conjugation invariants is conjugation invariant, for similar reasons as in \cite{bib:DR2018},
and obviously conjugation invariants form a linear space.
Hence, conjugation invariants form a (graded) subalgebra of $(T(\R^d),\shuffle)$.
\newcommand\conjInvariants{\mathsf{ConjInv}}
We write
\begin{align*}
    \conjInvariants = \bigoplus_{n \ge 0} \conjInvariants_n,
\end{align*}
for this graded subalgebra.

In \autoref{fig:conj}, two two-dimensional paths which are conjugate are shown.
It is evident that the signed area of these two paths differs, hence $\frac12(\word{12}-\word{21})$ is \emph{not} a conjugation invariant.

\tikzset{every picture/.style={line width=0.75pt}} 

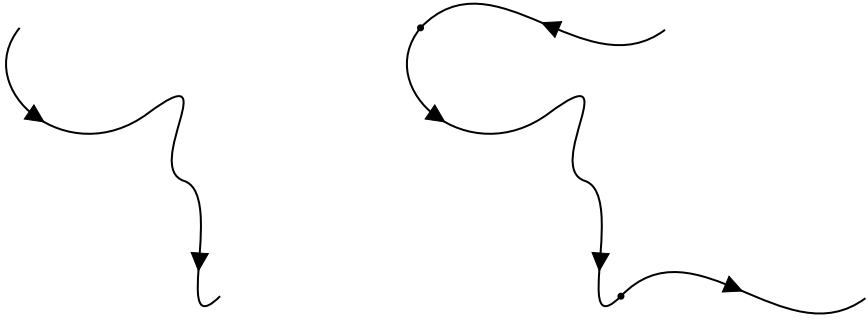
\begin{figure}
\begin{tikzpicture}[x=0.75pt,y=0.75pt,yscale=-1,xscale=1]

\draw    (328,56.5) .. controls (365,17.5) and (410,87.5) .. (450,57.5) ;
\draw [shift={(388.5,53.89)}, rotate = 22.76] [fill={rgb, 255:red, 0; green, 0; blue, 0 }  ][line width=0.08]  [draw opacity=0] (8.93,-4.29) -- (0,0) -- (8.93,4.29) -- cycle    ;

\draw    (428,191.5) .. controls (465,152.5) and (510,222.5) .. (550,192.5) ;
\draw [shift={(488.5,188.89)}, rotate = 202.76] [fill={rgb, 255:red, 0; green, 0; blue, 0 }  ][line width=0.08]  [draw opacity=0] (8.93,-4.29) -- (0,0) -- (8.93,4.29) -- cycle    ;
\draw    (328,56.5) .. controls (303,87.5) and (352,129.5) .. (392,99.5) .. controls (432,69.5) and (389,126.5) .. (410,133.5) .. controls (412.03,134.18) and (413.6,135.49) .. (414.8,137.3) .. controls (416.8,140.3) and (417.8,144.65) .. (418.21,149.65) .. controls (420.12,172.85) and (409.42,210.08) .. (428,191.5) ;
\draw [shift={(340.1,103.82)}, rotate = 214.11] [fill={rgb, 255:red, 0; green, 0; blue, 0 }  ][line width=0.08]  [draw opacity=0] (8.93,-4.29) -- (0,0) -- (8.93,4.29) -- cycle    ;
\draw [shift={(417.2,178.7)}, rotate = 274.15999999999997] [fill={rgb, 255:red, 0; green, 0; blue, 0 }  ][line width=0.08]  [draw opacity=0] (8.93,-4.29) -- (0,0) -- (8.93,4.29) -- cycle    ;

\draw [fill] (328,56.5) circle (.2ex);
\draw [fill] (428,191.5) circle (.2ex);

\begin{scope}[xshift=-150]
\draw    (328,56.5) .. controls (303,87.5) and (352,129.5) .. (392,99.5) .. controls (432,69.5) and (389,126.5) .. (410,133.5) .. controls (412.03,134.18) and (413.6,135.49) .. (414.8,137.3) .. controls (416.8,140.3) and (417.8,144.65) .. (418.21,149.65) .. controls (420.12,172.85) and (409.42,210.08) .. (428,191.5) ;
\draw [shift={(340.1,103.82)}, rotate = 214.11] [fill={rgb, 255:red, 0; green, 0; blue, 0 }  ][line width=0.08]  [draw opacity=0] (8.93,-4.29) -- (0,0) -- (8.93,4.29) -- cycle    ;
\draw [shift={(417.2,178.7)}, rotate = 274.15999999999997] [fill={rgb, 255:red, 0; green, 0; blue, 0 }  ][line width=0.08]  [draw opacity=0] (8.93,-4.29) -- (0,0) -- (8.93,4.29) -- cycle    ;
\end{scope}
\end{tikzpicture}
\caption{Two conjugate paths.
  Their signatures will agree on the conjugation invariants.
  It can clearly be seen that their signed areas are different, so $\word{12}-\word{21}$ is not a conjugation invariant.}
\label{fig:conj}
\end{figure}

\begin{definition}
    Given a word $w$, the element $\rot(w)\in T(\R^d)$ is the sum of its ``rotations'', i.e. cyclic permutations.
\end{definition}
\begin{remark}
This is related to
Poincar\'e's concept of \emph{equipollent}, \cite[Theorem 3.3]{bib:TT1999}.
\end{remark}

For example
\begin{align*}
\rot(\word{1})&=\word{1}\\
  \rot(\word{11})&=2\,\word{11}\\
  \rot(\word{12})&=\word{12}+\word{21}\\
  \rot(\word{122})&=\word{122}+\word{212}+\word{221}\\
  \rot(\word{1212})&=2(\word{1212}+\word{2121})
\end{align*}
Note that for a word $w$, all the words occuring in $\rot{w}$ have the same coefficient, which is $\rep(w)$, the number of times $w$ contains its smallest repeating pattern.
\begin{align*}
  \rep(w):=\max\left\{n\mid\exists v:\;w=v^n\right\}
\end{align*}

\begin{proposition}
  \label{prop:rot}
  Given a word $w$, $\rot(w)$ is a conjugation invariant.
\end{proposition}
\begin{proof}
  Let $A$ and $B$ be paths.
  Chen's identity
  (\cite[Theorem 7.11]{friz2010multidimensional}) states that for any word w,
  \begin{align}
    \left<\sig(A\sqcup B), w\right>=\sum_{uv=w}\left<\sig(A),u\right>\left<\sig(B),v\right>
  \end{align}

  Let $w$ be a word.
  If the word $v$ occurs in $\rot(w)$, then any rotation of $v$ will as well.
  \begin{align*}
    \left<\sig(A\sqcup B), \rot(w)\right>
    &=\rep(w)\sum_{uv\text{ occurs in }\rot(w)}\left<\sig(A),u\right>\left<\sig(B),v\right>
    \\&=\rep(w)\sum_{vu\text{ occurs in }\rot(w)}\left<\sig(A),u\right>\left<\sig(B),v\right>
    \\&=\left<\sig(B\sqcup A), \rot(w)\right>\qedhere
  \end{align*}
\end{proof}

\begin{remark}
    If $v$ and $w$ are words, then $v$ occurs in $\rot(w)$ if and only if $\rot(v)=\rot(w)$.
    Thus the relation given by
    \begin{align*}
      v\sim w\quad\text{if}\quad v\text{ occurs in }\rot(w)
    \end{align*}
    is an equivalence relation.
    Values of the form $\rot(w)$ for words $w$ of length $m$ correspond to $m$-ary necklaces (\cite[Chapter 7]{reutenauer1993free}).
Their count is the sum of the dimensions of the graded parts
    of the free Lie algebra, $\FL_k(\R^d)$,
    with $k$ dividing $m$, which correspond to aperiodic necklaces.
\end{remark}

\begin{table}
\centering
\begin{tabular}{c||cccc}
\hline
Level&\shortstack{signature\\elements}&\shortstack{logsignature\\elements\\\oeis{A001037}}&\shortstack{conjugation\\invariants\\\oeis{A000031}}&\shortstack{minimal $\shuffle$-generators\\of conj. inv.}\\
\hline
1&2&2&2&2 \hphantom{(2)}\\
2&4 (6)& 1 (3)& 3 (5) & 0 (2)\\
3&8 (14)& 2 (5)&4 (9) & 0 (2)\\
4&16 (30)&3 (8)&6 (15) & 1 (3)\\
5&32 (62)&6 (14)&8 (23) &0 (3)\\
6&64 (126)&9 (23)&14 (37)&4 (7)\\
\hline
\end{tabular}
\caption{Sizes of various sets for two-dimensional paths, along with cumulative sums, ignoring level 0.
Up to level $6$ the algebra appears free (the last column is equal to the Euler transform
of the first column).
}
\end{table}
\begin{table}
\centering
\begin{tabular}{c||cccc}
\hline
Level&\shortstack{signature\\elements}&\shortstack{logsignature\\elements\\\oeis{A027376}}&\shortstack{conjugation\\invariants\\\oeis{A001867}}&\shortstack{minimal $\shuffle$-generators\\of conj. inv.}\\
\hline
1&3&3&3&3\\
2&9 (12)& 3 (6)& 6 (9) & 0 \\ 3&27 (39)& 8 (14)&11 (20) & 1 \\ 4&81 (120)&18 (32)&24 (44) & 6\\
5&243 (363)&48 (80)&51 (95) & 6\\
6&729 (1092)&116 (196)&130 (225) & 38\\
\hline
\end{tabular}
\caption{
Sizes of various sets for three-dimensional paths, along with cumulative sums, ignoring level 0.
The algebra is not free: if it were,
the $38$ in the last row would be a $37$.
}
\label{tab:conj_d3}
\end{table}

\subsection{Characterizing conjugation invariants}

\begin{proposition}
    \label{prop:conj_invariants}
For an element $\EL\in T(\R^d)$
    the following are equivalent
    \begin{enumerate}
        \item $\EL$ is conjugation invariant.
        
        \item
        For all grouplike elements $g,h$
        \begin{align*}
            \langle h g h^{-1}, \EL \rangle = \langle g, \EL \rangle.
        \end{align*}
        
        \item
        For all Lie polynomials $L$ and all elements $q \in T_{\ge 1}(\R^d)$
        \begin{align}
            \langle [L, q], \EL \rangle = 0.
        \end{align}

    \item
      For every letter $\word{i}$
      and all elements
      $q \in T_{\ge 1}(\R^d)$
      \begin{align}
          \langle [\word{i}, q], \EL \rangle = 0.
      \end{align}    
    \end{enumerate}    
\end{proposition}
\begin{proof}
  ``1. $\Leftarrow$ 2.'':
  This follows from the fact that
  \begin{align*}
      \sig(B\sqcup A\sqcup B^{-1})
      =
      \sig(B) \sig(A) \sig(B)^{-1},
  \end{align*}
  and the fact that $\sig(B)$ is grouplike.
  
  ``1. $\Rightarrow$ 2.'':
  By assumption, there is some $n \ge 1$
  such that 
  \begin{align}
  \label{eq:XisN}
    \EL \in T_{\le n}(\R^d).
  \end{align}
  Let $g, h$ be given grouplikes.
  By the Chow/Chen theorem (\cite[Theorem 7.28]{friz2010multidimensional}), there exist piecewise smooth paths
  $B, A$ with
  \begin{align*}
      \pi_{\le n} \sig(B) &= \pi_{\le n} h \\
      \pi_{\le n} \sig(A) &= \pi_{\le n} g.
  \end{align*}
  Then, using \eqref{eq:XisN} repeatedly,
  \begin{align*}
    \langle h g h^{-1}, \EL \rangle
    &=
    \langle \pi_{\le n}( h g h^{-1} ), \EL \rangle \\
    &=
    \langle \pi_{\le n}( h )
    \pi_{\le n}(g) \pi_{\le n}(h^{-1} ), \EL \rangle \\
    &=
    \langle \pi_{\le n}( \sig(B) )
    \pi_{\le n}(\sig(A)) \pi_{\le n}(\sig(B)^{-1} ), \EL \rangle \\
    &=
    \langle \sig(B)
    \sig(A)\sig(B)^{-1}, \EL \rangle \\
    &=
    \langle \sig(A), \EL \rangle \\
    &=
    \langle g, \EL \rangle.
  \end{align*}
  
  $1. \Rightarrow 3.$:
  Let $L, g$ be given.
  Then, by 2.,
  for any $t \in \R$
  \begin{align*}
    \langle \exp(t L) g \exp(-t L), \EL \rangle = \langle g, \EL \rangle.
  \end{align*}
  Taking the derivative with respect to $t$, at $t=0$, we deduce
  \begin{align*}
    \langle L g - g L, \EL \rangle = 0,
  \end{align*}
  for any grouplike element $g$.
Since grouplike elements span
  (level by level) the space $T_{\ge 1}((\R^d))$ with
  zero constant term (\cite[Lemma 3.4]{bib:DR2018}) the result follows.
  
  $1. \Leftarrow 3.$:
  Let $h,g$ grouplike be given.
  By 2.
  it suffices to show
  \begin{align*}
      \langle h g h^{-1}, \EL \rangle = \langle g, \EL \rangle.
  \end{align*}
  Let $L := \log h$ and consider
  \begin{align*}
      f(t) := \langle \exp(t L) g \exp(- t L), \EL \rangle.
  \end{align*}
  Then
  \begin{align*}
      f'(t) =
      \langle L \exp(t L) g \exp(- t L)
                - \exp(t L) g \exp(- t L) L, \EL \rangle
        =:
      \langle L g_t - g_t L, \EL \rangle,
  \end{align*}
  with $g_t$ grouplike.
  Hence, by assumption, $f'(t) = 0$ for every $t$.
  Hence
  \begin{align*}
      \langle g, \EL \rangle = f(0) = f(1) = \langle h g h^{-1}, \EL \rangle.
  \end{align*}

  The equivalence of the final statement follows from \Cref{lem:LieFLA}.
\end{proof}

\begin{lemma}
  \label{lem:LieFLA}
  Let $\FL(\R^d)$ be the free Lie algebra over $\R^d$.
  Then
  \begin{align*}
    [ \FL(\R^d), T(\R^d) ]
    =
    [ \R^d, T(\R^d) ].
  \end{align*}    
\end{lemma}

\begin{proof}
  Any $L$ can be written as a linear combination of
  ''right-combs'' $[\word{i_1},[\word{i_2},\dots [\word{i_{n-1}},\word{i_n}]\dots]]$,
  and for such an element and $\phi \in T(\R^d)$, by the Jacobi identity,
\begin{align*}
[ [\word{i_1},[\word{i_2},\dots [\word{i_{n-1}},\word{i_n}]\dots]], \phi ]
    =
    [ [\word{i_2},\dots [\word{i_{n-1}},\word{i_n}]\dots], [\phi, \word{i_1}] ]
    +
    [ [\phi, [\word{i_2},\dots [\word{i_{n-1}},\word{i_n}]\dots]], \word{i_1} ].
  \end{align*}
  The second term is in
  $[T(\R^d), \R^d ] = [\R^d, T(\R^d)]$
  and the first is in
$[\FL_{\le n-1}(\R^d), T(\R^d)]$
  where $\FL_{\le n-1}(\R^d) := \FL(\R^d) \cap T_{\le n}(\R^d)$.
  By induction the claim follows.
\end{proof}

This leads to another proof of
\Cref{prop:rot}.

\begin{proposition}
  For any word $w$,  for any $M,N \in T((\R^d))$
  \begin{align*}
      \langle [M,N], \rot(w) \rangle = 0.
  \end{align*}
  In particular: $\rot(w)$ is a conjugation invariant.
\end{proposition}
\begin{proof}
    It is enough to verify this for
    $M$ a word of length $m$,
    $N$ a word of length $n$
    with $m+n = |w|$,
    and $MN$ (and then also $NM$) being a cyclic rotation
    of $w$.
    In this case
    \begin{align*}
       \langle MN, \rot(w) \rangle = 1 = \langle NM, \rot(w) \rangle, 
    \end{align*}
    which proves the first claim.

    The second claim then follows from
\Cref{prop:conj_invariants}
\end{proof}

\begin{definition}
  Let $\shift_m$ be the automorphism of words of length $m$, which performs a rotation taking the last letter to the front.
  This is extended linearly to an automophism of level $m$ of the tensor algebra.
  For example $\shift_3(\word{231})=\word{123}$ and
  \begin{align*}
    \shift_5(\word{12345}+3\,\word{12344})=\word{51234}+3\,\word{41234}
  \end{align*}
\end{definition}

\begin{proposition}
  \label{prop:conjugationShift2}
  Let $\EL$ be a conjugation invariant at level $m$.
  Then $\EL=\shift_m(\EL)$.
\end{proposition}
\begin{proof}
   By
\Cref{prop:conj_invariants},
   for any letter $i$ and any
   $\phi$ having zero constant term
   \begin{align}
     \label{eq:iphi}
     \langle i\phi - \phi i, \EL \rangle = 0.
   \end{align}
   Write $\EL = \sum c_w w$ for some coefficents $c_w \in \R$.
   Take $w=w_1 \dots w_m$ any word and let $\phi = w_2 \dots w_m$, $i= w_1$.
   Then \eqref{eq:iphi} implies
   \begin{align*}
       c_w = c_{w_2 \dots w_m w_1}.
   \end{align*}
   Hence $\EL = \shift_m(\EL)$.
\end{proof}
\begin{theorem}
  \begin{align*}
       \conjInvariants = \operatorname{im} \rot.
   \end{align*} 
\end{theorem}
\begin{proof}

    $\supseteq$: \Cref{prop:rot}.
    
    $\subseteq$:
    Using  \Cref{prop:conjugationShift2}
    it remains to show that for $\EL \in T_m$,
    \begin{align*}
        \EL \in \im \rot \Leftrightarrow \EL = \shift_m(\EL).
    \end{align*}
    And indeed, ``$\Rightarrow$'' is immediate.
    Regarding ``$\Leftarrow$'',
    we only have to note for $\EL \in \pi_m \im \rot$
    \begin{align*}
        \EL = \frac1m \sum_{i=0}^{m-1} \shift_m(\EL).
    \end{align*}
\end{proof}

\subsection{Examples}
\subsubsection{Steps}
The first conjugation invariants in two dimensions which is not just a shuffle of letters is at level 4.
A representative new invariant is $\EL = \word{1212}+\word{2121}$.
It is instructive to consider a piecwise linear path as an illustration.
Take the path $\gamma$ which travels $a$ to the right, $b$ up, $c$ to the right and $d$ up as shown in \autoref{fig:steps}.
In this case $\left<\sig(\gamma),\word{2121}\right>=0$ and $\left<\sig(\gamma),\word{1212}\right>= \left<\sig(\gamma),\EL\right>=abcd$.
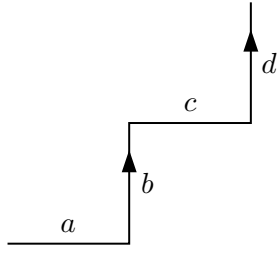
\begin{figure}
\centering
\begin{tikzpicture}[scale=0.8]
  \draw (0,0) -- (2,0) -- (2,2) -- (4,2) -- (4,4);
  \node at (1,0.3) {$a$};
  \node at (2.3,1) {$b$};
  \node at (3,2.3) {$c$};
  \node at (4.3,3) {$d$};
  \draw [shift={(2,1.5)}, fill]  [draw opacity=0] (0.1,-0.3) -- (0,0) -- (-0.1,-0.3) -- cycle    ;
  \draw [shift={(4,3.5)}, fill]  [draw opacity=0] (0.1,-0.3) -- (0,0) -- (-0.1,-0.3) -- cycle    ;
\end{tikzpicture}
\caption{The path $\gamma$ which consists of four axis-aligned line segments.}
  \label{fig:steps}
\end{figure}

\subsubsection{Signed volume}
For paths in more than two dimensions, the first conjugation invariants which are not just shuffles of letters occur at level 3, in the homogeneity class(es) which contain one each of three distinct letters.
Specifically, using the first three letters, $\rot(\word{123})$ and $\rot(\word{132})$ are distinct conjugation invariants, but only their sum $\rot(\word{123})+\rot(\word{132})$ is a shuffle of letters.
Their difference, the new conjugation invariant $\rot(\word{123})-\rot(\word{132})$, is the signed volume as described in \cite{bib:DR2018}.

\subsection{Shuffle structure}

\begin{definition}
We call a finite subset $\{s_1,\dots,s_m\}\subset T(\R^d)$ algebraically independent, or shuffle independent, if there is no polynomial $P\in\R[x_1,\dots,x_m]$ other than $P=0$ such that $P_{\shuffle}(s_1,\dots,s_m)=0$, where $P_{\shuffle}$ is such that each multiplication in $P$ is replaced by a shuffle product.
An infinite subset $S$ of $T(\R^d)$ is called algebraically independent, or shuffle independent, if each finite subset of $S$ is shuffle independent.
\end{definition}

\begin{proposition}\label{prop:infmanyconjinv}
 If $d\ge2$, there are infinitely many shuffle-independent conjugation invariants in $T(\R^d)$.
\end{proposition}
\begin{proof}
Let $n$ be arbitrary. We look at the axis parallel path with $n$ steps $x_1,\dots,x_n$ in $x$ direction and $n$ steps of size $1$ in $y$ direction. Under the signature, the conjugation invariants 
\begin{equation*}
\rot(\word{1}^{\bullet m+1}(\word{21})^{\bullet n-1}\word{2})
\end{equation*}
evaluate to 
\begin{equation*}
\frac{1}{n!}x_1\cdots x_n\cdot \sum_{i=1}^n x_i^m
\end{equation*}
By algebraic independence of $(\sum_{i=1}^n x_i^m)_{m=1}^n$, we obtain that there are at least $n$ shuffle-algebraically-independent conjugation invariants.

Since $n$ was arbitrary, by the basis exchange property for algebraic independence (see for example \cite[Chapter~9]{milneFT}\footnote{Algebraic independence is traditionally formulated in the language of field extensions. 
For this, consider the field of fractions of the shuffle algebra $(T(\mathbb{R}^d),\shuffle)$ as a field extension of $\mathbb{R}$.}), we thus conclude that there exists a set of infinitely many shuffle-algebraically-independent conjugation invariants.
\end{proof}

\begin{remark}
   \label{rem:conj_not_free}
   
   For dimension $2$ and up, the conjugation invariants are \emph{not} 
   free as a commutative algebra.
   Indeed, this can be verified by dimension counting.
   
   For dimension $2$ a relation appears on level $12$, see
   \Cref{tab:2dloopinv}.

   For dimension $3$ and $4$ we provide explicit relations.
   For dimension $3$, there is a relation
   on level $6$
   (see  \Cref{tab:conj_d3}), given by
   \begin{align*}
       0 &=        2         \cdot   \w1\shuffle\w1\shuffle\rot(\word{2233})
    +   2         \cdot   \w1\shuffle\w2\shuffle\rot(\word{1323})
        -       \cdot   \w2\shuffle\w2\shuffle\rot(\word{1313}) \\
       &\quad -4       \cdot   \w1\shuffle\w3\shuffle\rot(\word{1223})
        -4       \cdot   \w2\shuffle\w3\shuffle\rot(\word{1132})
    +    2       \cdot   \w3\shuffle\w3\shuffle\rot(\word{1122}) \\
    &\quad +    2       \cdot   \rot(\word{132})\shuffle\rot(\word{132})
        -2       \cdot   \w1\shuffle\w2\shuffle\w3\shuffle\rot(\word{132})
    +      \w1\shuffle\w1\shuffle\w2\shuffle\w2\shuffle\w3\shuffle\w3.
   \end{align*}

   For dimension $4$,
   the following non-trivial relation from \cite[p.496]{adin2021cyclic}
   (there, for certain cyclic quasisymmetric functions)
   \begin{align*}
        \word{1}\shuffle \vol_3(\word2,\word3,\word4) 
      - \word{2}\shuffle \vol_3(\word1,\word3,\word4) 
      + \word{3}\shuffle \vol_3(\word1,\word2,\word4) 
      - \word{4}\shuffle \vol_3(\word1,\word2,\word3),
   \end{align*}
   holds, where
   \begin{align*}
       \vol_3(a,b,c)
       &:= \rot(abc) - \rot(bac) \\
       &= abc + bca + cab - bac - acb - cba
   \end{align*}
   is the three-dimensional signed volume in the $a-b-c$ subspace.
   This relation is comprehensible through the presentation of $vol_3$ as
   \begin{equation*}
       vol_3(a,b,c)=a\shuffle(bc-cb)-b\shuffle(ac-ca)+c\shuffle(ab-ba)
   \end{equation*}
   
\end{remark}

\section{Loop invariants}
\label{sec:loop_invariants}

Loops form an important class of curves, 
in which the start- and endpoint coincide.
They appear
in
topology \cite{hatcher2002algebraic}
and
data science (contours of 2D objects, periodic time series \cite[Section 2.1.4]{bib:KBTdS2014}).

We are now interested in identifying the signature elements which are invariant for loops, i.e.~closed paths, upon changing their starting points.
Two closed paths $A$ and $B$ are said to \DEF{differ only by their starting point} if there exist paths $C$ and $D$ such that $A$ and $C\sqcup D$ are the same path and $B$ and $D\sqcup C$ are the same path.
An example is shown in \autoref{fig:loops}.
\begin{figure}
\centering
\begin{tikzpicture}[scale=0.022, yscale=-1]
  \draw   (155,111.5) .. controls (143,126.5) and (156,160.5) .. (192,179.5) .. controls (228,198.5) and (240,189.5) .. (267,192) .. controls (294,194.5) and (328,185.5) .. (350,140.5) .. controls (372,95.5) and (355,91.5) .. (343,90.5) .. controls (331,89.5) and (290,85.5) .. (278,101.5) .. controls (266,117.5) and (271,137.5) .. (251,133.5) .. controls (231,129.5) and (243,112.5) .. (218,101.5) .. controls (193,90.5) and (167,96.5) .. (155,111.5) -- cycle ;
\draw[fill] (155,111.5) circle (2);
  \node at (145, 105) {$a$};
  \node at (182, 190) {$A$};
  \draw [shift={(192,179.5)}, rotate = 210] [fill={rgb, 255:red, 0; green, 0; blue, 0 }  ][line width=0.08]  [draw opacity=0] (8.93,-4.29) -- (0,0) -- (8.93,4.29) -- cycle    ;
  \begin{scope}[shift={(300,0)}]
  \draw   (155,111.5) .. controls (143,126.5) and (156,160.5) .. (192,179.5) .. controls (228,198.5) and (240,189.5) .. (267,192) .. controls (294,194.5) and (328,185.5) .. (350,140.5) .. controls (372,95.5) and (355,91.5) .. (343,90.5) .. controls (331,89.5) and (290,85.5) .. (278,101.5) .. controls (266,117.5) and (271,137.5) .. (251,133.5) .. controls (231,129.5) and (243,112.5) .. (218,101.5) .. controls (193,90.5) and (167,96.5) .. (155,111.5) -- cycle ;
  \draw[fill] (278,101.5) circle (2);
  \draw [shift={(192,179.5)}, rotate = 210] [fill={rgb, 255:red, 0; green, 0; blue, 0 }  ][line width=0.08]  [draw opacity=0] (8.93,-4.29) -- (0,0) -- (8.93,4.29) -- cycle    ;
  \node at (289, 107) {$b$};
  \node at (182, 190) {$B$};
  \end{scope}
\end{tikzpicture}
\caption{Two paths $A$ and $B$ which are loops differing only by starting point.
$A$ is the concatenation of a path $C$ from $a$ to $b$ and a path $D$ from $b$ to $a$, while $B$ is the concatenation of $D$ with $C$.}
  \label{fig:loops}
\end{figure}
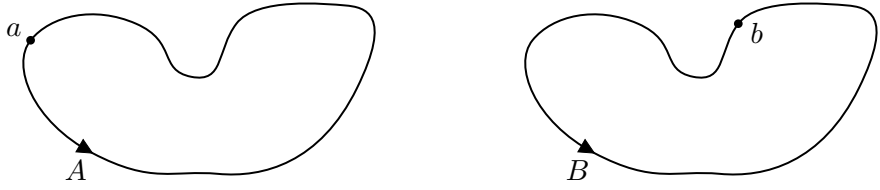

\begin{definition}
  A signature element $\EL\in T(\R^d)$ is called a \DEF{loop invariant} if for any two closed paths $A$ and $B$ which only differ by their starting point, $\left<\sig(A),\EL\right> = \left<\sig(B),\EL\right>$.
  
    A signature element $\EL\in T(\R^d)$ is called a \DEF{conjugation invariant for loops} if for any closed path $A$ and any path $B$, $\left<\sig(A),\EL\right> = \left<\sig(B\sqcup A\sqcup B^{-1}),\EL\right>$.
\end{definition}

The next statement shows that the two notions are equivalent.

\begin{proposition}
  \label{prop:conj_for_loops_implies_loop}
  \label{prop:loop_implies_conj_for_loops}
  $\EL$ is a conjugation invariant for loops if and only if $\EL$ is a loop invariant.
\end{proposition}
\begin{proof}~
    $\Rightarrow$:
   (The basic idea: notice that $A$ and $B$ in \autoref{fig:loops} are conjugate paths.)
    
  Let $A$ and $B$ be two closed paths which differ only by starting point.
  Then there exist paths $C$ and $D$ such that $\sig(A)=\sig(C\sqcup D)$ and $\sig(B)=\sig(D\sqcup C)$.
  Then $\sig(B)=\sig(D\sqcup A\sqcup D^{-1})$.
  \begin{align*}
    \left<\sig(B),\EL\right>&=\left<\sig(D\sqcup A\sqcup D^{-1}),\EL\right>
                          =\left<\sig(A),\EL\right>
  \end{align*}
  where the last step used that $\EL$ is a conjugation invariant for loops.

  $\Leftarrow$:
  (The basic idea may be illustrated by \autoref{fig:tadpole}.)
  
  Let $A$ be a closed path and $B$ be a path.
  We observe that $B\sqcup A\sqcup B^{-1}$ is a closed path.
  We further observe that $B\sqcup A\sqcup B^{-1}$ and $A\sqcup B^{-1}\sqcup B$ are closed paths differing only by starting point.
  Therefore
  \begin{align*}
    \left<\sig(B\sqcup A\sqcup B^{-1}),\EL\right> &= \left<\sig(A\sqcup B^{-1}\sqcup B),\EL\right>
    =\left<\sig(A),\EL\right>
  \end{align*}
  where the last step used the fact that two tree-equivalent paths have the same signature.
  
\end{proof}

\begin{figure}
\centering
\begin{tikzpicture}[scale=0.022, yscale=-1]
  \draw   (155,111.5) .. controls (143,126.5) and (156,160.5) .. (192,179.5) .. controls (228,198.5) and (240,189.5) .. (267,192) .. controls (294,194.5) and (328,185.5) .. (350,140.5) .. controls (372,95.5) and (355,91.5) .. (343,90.5) .. controls (331,89.5) and (290,85.5) .. (278,101.5) .. controls (266,117.5) and (271,137.5) .. (251,133.5) .. controls (231,129.5) and (243,112.5) .. (218,101.5) .. controls (193,90.5) and (167,96.5) .. (155,111.5) -- cycle ;
  \draw (155,111.5) .. controls(120,80) and (60,170) .. (20,120);
\draw[fill] (155,111.5) circle (2);
  \draw[fill] (20,120) circle (2);
  \node at (182, 190) {$A$};
  \node at (50,150) {$B$};
  \node at (105,100) {$B^{-1}$};
  \draw[->] (61,150) -- (80, 145);
  \draw[->] (85,105) -- (65, 115);
  \draw [shift={(192,179.5)}, rotate = 210] [fill={rgb, 255:red, 0; green, 0; blue, 0 }  ][line width=0.08]  [draw opacity=0] (8.93,-4.29) -- (0,0) -- (8.93,4.29) -- cycle    ;
\end{tikzpicture}
\caption{A closed path $A$ conjugated by a path $B$.}
  \label{fig:tadpole}
\end{figure}
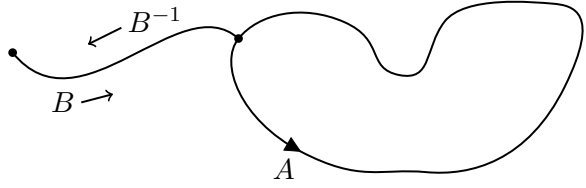

\newcommand\loopInvariants{\mathsf{LoopInv}}
We denote
\begin{align*}
    \loopInvariants = \bigoplus_{n\ge 0} \loopInvariants_n,
\end{align*}
the graded subalgebra of loop invariants
(which, by the previous statements is the same
as the subalgebra of conjugation invariants for loops).

Clearly
\begin{align*}
    \conjInvariants \subset \loopInvariants,
\end{align*}
and the inclusion is strict for $d\ge 2$, since for example
the area 
$\w1\w2-\w2\w1$ is a loop invariant but not a conjugation invariant.

\subsection{Characterizing loop invariants}

Analogously to the previous subsection, we get the following equivalence.
\begin{proposition}
    \label{prop:loop_invariants}
    For an element $\EL\in T(\R^d)$
    the following are equivalent
    \begin{enumerate}
        \item $\EL$ is loop invariant
        \item 
        For all grouplikes $h,g$ satisfying $\langle g, i \rangle = 0$
        for all letters $i = 1, \dots, d$,
        we have
        \begin{align*}
        \langle h g h^{-1}, \EL \rangle = \langle g, \EL \rangle.
        \end{align*}
        \item
        For all Lie polynomials $L$ and all grouplikes $g$ satisfying $\langle g, i \rangle = 0$
        for all letters $i = 1, \dots, d$,
        we have
        \begin{align*}
            \langle [L, g], \EL \rangle = 0.
        \end{align*}
    \end{enumerate}
\end{proposition}
Another characterization of loop invariants
will be given in \Cref{thm:vi}.

We now consider
the linear span of those grouplike $g$ satisfying $\langle g, i \rangle = 0$ for all letters $i = 1, \dots, d$,
used in \Cref{prop:loop_invariants}.ii.

\begin{definition}
    Let
    \begin{align*}
        V &:= \bigoplus_{n\ge 0} V_n \\
        V_n &:= \operatorname{span}_\R \{ \pi_n g \mid g\text{ grouplike},  
        \langle g, \word{i} \rangle = 0 \text{ for all letters } i = 1, \dots, d \},
    \end{align*}
    the ``linear span of grouplike elements with zero increment''
    and
    \begin{align*}
        S   := \bigoplus_{n\ge 1} S_n, \qquad
        S_n :=
        \operatorname{span}_\R\{ \word{i} \shuffle T_{n-1}(\R^d) \mid i =  1, \dots, d \},
    \end{align*}
    the shuffle ideal generated by the letters.
\end{definition}

\begin{lemma}
    \label{lem:VnSn}
    With respect to the usual inner product
    (which declares words to be orthogonal)
    \begin{align*}
        V_n^\bot = S_n, \quad V_n = S_n^\bot.
    \end{align*}
\end{lemma}
\begin{proof}
  Since both $V_n$ and $S_n$ are finite-dimensional, and hence closed, subspaces,
  it is enough to show one of the equalities. We show
  the second.

  Clearly 
  \begin{align}
    \label{eq:inclusion}
    V_n \subset S_n^\bot,
  \end{align}
  since for every $g$ with zero increments,
  any letter $i$ and any word $w$ with $|w| = n-1$,
  \begin{align*}
    \langle \pi_n g, \word{i} \shuffle w \rangle
    =
    \langle g, \word{i} \shuffle w \rangle
    =
    \langle g, \word{i} \rangle
    \cdot_\R
    \langle g, w \rangle
    =
    0.
  \end{align*}
We will now show $T_n(\R^d) = S_n \oplus R_n$
  for some space $R_n$, that satisfies
  \begin{align*}
    \forall r \in R_n \setminus \{0\} : \exists v \in V_n: \langle v, r \rangle \not= 0.
  \end{align*}
  This gives 
  \begin{align*}
    \dim V_n \ge \dim R_n = \dim T_n - \dim S_n = \dim S_n^\bot.
  \end{align*}
  This, together with \eqref{eq:inclusion}, implies $V_n = S_n^\bot$.

  It remains to construct $R_n$.
  Consider the Lyndon basis of the free Lie algebra
  \begin{align*}
    P_h, \quad h \text{ Lyndon word},
  \end{align*}
  and the corresponding coordinates of the first kind
  \begin{align*}
    \zeta_h, \quad h \text{ Lyndon word},
  \end{align*}
  satisfying
  \begin{align*}
    \langle \exp( \sum_h c_h P_h ), \zeta_w \rangle = c_w.
  \end{align*}
  It is known, \cite{bib:DLPR2020}, that the coordinates
  of the first kind (freely) shuffle-generate the tensor algebra.
Let
  \begin{align*}
    R_n := \pi_n \operatorname{span}_\R \{ \text{shuffle monomials in non-letter words} \}.
  \end{align*}
  Then
  \begin{align*}
    R_n = \pi_n \operatorname{span}_\R \{ \text{shuffle monomials in non-letter coordinates of the first kind} \}.
  \end{align*}
  Then clearly $T_n = S_n \oplus R_n$.
  For $\zeta_{h_1} \shuffle \dots\shuffle \zeta_{h_k}$
  a monomial in non-letter coordinates of the first kind (with total
  degree $n$) let
  \begin{align*}
    g := \exp( \sum_{i=1}^k P_{h_i} ).
  \end{align*}
  We then get
  \begin{align*}
    \langle \pi_n g, \zeta_{h_1} \shuffle \dots\shuffle \zeta_{h_k} \rangle
    =
    \langle g, \zeta_{h_1} \shuffle \dots\shuffle \zeta_{h_k} \rangle
    =
    \langle g, \zeta_{h_1} \rangle \cdot \ldots \cdot \langle g, \zeta_{h_k} \rangle
    =
    1,
  \end{align*}
  as desired
\end{proof}

\begin{proposition}
\label{prop:VnLiePoly}
    For $n \in \N$,
    \begin{align*}
    V_n =  
    \pi_n \left( \mathbb{R}\oplus\operatorname{span}\{L_1\cdots L_k|L_1,\ldots,L_k\text{ non-letter Lie polynomials}, k\in\mathbb{N}\} \right).
    \end{align*}
\end{proposition}
\begin{proof}
   Write $W_n$ for the right-hand side of the claimed
   identity.

  ``$\supseteq$'': 
Let $L_1,\dots,L_k$ be homogeneous, non-letter
  Lie elements. Let $n$ be the degree of $L_1 \cdot \ldots \cdot L_k$.
  Then
  \begin{align*}
      \pi_n 
      \exp( t_1 L_1 )
      \dots
      \exp( t_k L_k ) \in V_n,
  \end{align*}
  for all $t_1,\dots,t_k \in \R$,
  since $\langle \exp( t_j L_j ), \word{i} \rangle = 0$, for all $j=1,\dots,k, i=1,\dots,d$.
  
  The space $V_n$ is closed, hence
  \begin{align*}
      V_n \ni
      \frac{d}{dt_1}|_{t_1=0}
      \dots
      \frac{d}{dt_k}|_{t_k=0}
      \exp( t_1 L_1 )
      \dots
      \exp( t_k L_k )
      =
      L_1 \dots L_k,
  \end{align*}
  as desired.
  
We shall obtain a space $R_n$ such
  that $T_n = S_n \oplus R_n$ and such that
  \begin{align*}
      \forall r \in R_n\setminus\{0\}: \exists w \in W_n: \langle w, r \rangle \not= 0.
  \end{align*}
  Then
  \begin{align*}
      \dim W_n \ge \dim R_n = \dim S_n^\bot = \dim V_n,
  \end{align*}
  and, since $W_n \subseteq V_n$ is shown above, we get $W_n = V_n$.
  
  Let $L_i, i \in I$, be a homogeneous basis for the Lie algebra.
  Assume without loss of generality, that $\{1,\dots,d\}\subset I$ and
  $L_i = \word{i}$ for $i=1,\dots, d$.
Let $f_i, i \in I$, be a (homogeneous) realization of the dual basis inside of $T(\R^d)$.
  That is, $f_i \in T(\R^d)$ are homogeneous elements satisfying
  \begin{align*}
      \langle f_i, L_j \rangle = \delta_{i,j}, \qquad i,j\in I.
  \end{align*}
  It is well-known (see for example \cite[Corollary 4.5]{bib:DLPR2020})
  that the $f_i$ freely shuffle-generate all of $T(\R^d)$.
  Define
  \begin{align*}
      R_n := \pi_n \operatorname{span}_\R \{ \text{shuffle monomials in non-letter words} \}.
  \end{align*}
  Then
  \begin{align*}
      R_n := \pi_n \operatorname{span}_\R \{ \text{shuffle monomials in non-letter $f_i$} \}.
  \end{align*}
  As before, $T_n = S_n \oplus R_n$.
  Let $f_{i_1} \shuffle ... \shuffle f_{i_k} \in R_n$ be given
  and assume for simplicity that the $f_{i_\ell}$ are pairwise different
  (otherwise multiplicities appear in what follows).
  Then
  \begin{align*}
      \langle L_{i_1} \cdot \dots \cdot L_{i_k}, 
              f_{i_1} \shuffle ... \shuffle f_{i_k} \rangle
      &=
      \langle \Delta_{\shuffle}^{\otimes k} L_{i_1} \cdot \dots \cdot L_{i_k}, 
              f_{i_1} \otimes ... \otimes f_{i_k} \rangle \\
      &=
      \sum_{\sigma \in S_k}
      \langle L_{i_{\sigma(1)}},  f_{i_1} \rangle
      \dots
      \langle L_{i_{\sigma(k)}},  f_{i_k} \rangle \\
      &=
      \langle L_{i_{1}},  f_{i_1} \rangle
      \dots
      \langle L_{i_{k}},  f_{i_k} \rangle = 1,
  \end{align*}
  as desired.
 
\end{proof}

As immediate corollary of
\Cref{prop:VnLiePoly} we get the following result.
\begin{corollary}\label{cor:dimVn}
$V$ can be identified with the enveloping algebra of
  the first term in the derived series of the free Lie algebra
  $[\mathfrak g, \mathfrak g]$ (also called ``derived free Lie Algebra'').
  Further
  \begin{align*}
    \dim V_n = [ (1-q)^d / (1-d q) ]_{q^n},
  \end{align*}
  the coefficient of $q^n$ in the power series
  $(1-q)^d / (1-d q)$.
\end{corollary}
\begin{remark}
  $V$ can also be characterized
  as the  
  intersection
  \begin{align*}
      \bigcap_{i=1}^d \ker\partial_{\word{i}},
  \end{align*}
  where $\partial_a w$ is the sum of all words obtained by
  removing, one-by-one, the occurences of $a$ in $w$; $\partial_a$ is a derivation.
  See \cite[1.6.5 Derivations]{reutenauer1993free}
  \cite[Proposition 6.1]{bergeron2008invariants},
   \cite[Section 3]{briand2008sn}.
\end{remark}
\begin{example}
  $d=3$, then
  \begin{align*}
    (1-q)^3 / (1-3 q)
    =
    1 + 3q^2 + 8 q^3 + 24 q^4 + 72 q^5
    + 216 q^6
    + 648 q^7
    + 1944 q^8 + \dots,
\end{align*}
  and hence, for example $\dim V_1 = 0, \dim V_2 = 3, \dim V_8 = 1944$.
  This is OEIS sequence \oeis{A118264}.
\end{example}
\begin{proof}[Proof of \Cref{cor:dimVn}]

  First, 
  \begin{align*}
      [\frac{1}{1-dq}]_{q^n},
  \end{align*}
  is the dimension of $T_n(\R^d)$,
  which in a PBW basis is spanned by all monomials
  in a basis of the free Lie algebra with total degree $n$.
  From these monomials we have to remove all monomials
  that contain some letter $\word{i}$.
  There are
  \begin{align*}
      [\frac{1}{1-dq}]_{q^n}
      -
      [\frac{1}{1-dq}]_{q^{n-1}}
      =
      [\frac{1-q}{1-dq}]_{q^{n}},
  \end{align*}
  many monomials that do not contain $\word{1}$.
  Then, of the remaining monomials there are
  \begin{align*}
        [\frac{1-q}{1-dq}]_{q^{n}}
        -
        [\frac{1-q}{1-dq}]_{q^{n-1}}
        =
        [\frac{(1-q)^2}{1-dq}]_{q^{n}},
  \end{align*}
  many monomials that do not contain $\word{2}$.
  Iterating, we obtain the claimed
  generating function.
\end{proof}

\begin{theorem}
\label{thm:vi}
$\EL\in T(\R^d)$ is a loop invariant if and only if
\begin{equation}
\langle [v,\word{i}],\EL\rangle=0
\end{equation}
for all $v\in V$ and all letters $i$.
\end{theorem}
\begin{proof}
By \Cref{prop:conj_for_loops_implies_loop}
and \Cref{prop:loop_implies_conj_for_loops},
loop invariants are conjugation invariants for loops.
By the Chen-Chow theorem
(\cite[Theorem 7.28]{friz2010multidimensional}), $\langle hgh^{-1},\EL\rangle=\langle g,\EL\rangle$ for all grouplikes $h$ is equivalent to when we just consider the signatures of piecewise linear paths for $h$. But then since we consider all $g$ with $\langle g,\word{i}\rangle=0$, it is in fact enough to only consider linear paths for $h$, and by derivation and linearization this is equivalent to $\langle [v,\word{i}],\EL\rangle=0$ for all $v\in V$ and all letters $i$.
\end{proof}

The space of loop invariants is thus given by $[V,\R^d]^{\perp}$.
Elements of $S$ evaluate to zero on loops and hence, trivially, are
loop invariants.
\begin{definition}
We call the quotient vector space $[V,\R^d]^{\perp}/S$ the space of 
\DEF{letter-reduced loop invariants}.
\end{definition}

\textbf{Now: we do not know whether there are ``non-trivial'' loop invariants apart from (shuffles of) the area and conjugation invariants.}

\begin{conjecture}\label{conj:loopinvareSplusareaconj}
 We have $\loopInvariants=S+\areaconj$,
 where $\areaconj$ is the shuffle subalgebra of $T(\R^d)$ generated by $(\word{ij}-\word{ji})_{i,j=1\dots d}$ and $\conjInvariants$.
\end{conjecture}

It might help to do dimension counting directly in the sub Hopf algebra of loops $V$ and its graded dual to not be bothered by the trivial loop invariants.

\begin{lemma}
    \begin{align*}
       \text{$[v,\word{i}]\in V$ for all letters $i$ and $v\in V$}
    \end{align*}
\end{lemma}
\begin{proof}
$v\in V$, $i$ a letter,
$w = w_1 \dots w_k$ any word, $j$ a letter.
Then
\begin{align*}
   \langle v\word{i}, w\shuffle \word{j} \rangle 
   &=
   \langle v\otimes \word{i}, \Delta_\bullet w \shuffle \Delta_\bullet \word{j} \rangle \\
   &=
   \langle v\otimes \word{i}, \Delta_\bullet w \shuffle (\e\otimes \word{j} + \word{j} \otimes \e) \rangle \\
   &=
   \langle v \otimes i, w \otimes \word{j} \rangle
   +
   \langle v \otimes i, w_1 \dots w_{k-1} \shuffle j \otimes w_k \rangle \\
   &=
   \langle v \otimes \word{i}, w \otimes \word{j} \rangle.
\end{align*}
Analogously, 
\begin{align*}
    \langle \word{i}v, w\shuffle \word{j} \rangle 
    =
   \langle \word{i} \otimes v, \word{j} \otimes w \rangle,
\end{align*}
and the claim follows.
    
\end{proof}

We can then write the
dimension of the letter-reduced loop invariants 
as
\begin{align*}
 \dim([V_{n-1},\R^d]^{\perp}/S_n)&=\dim([V_{n-1},\R^d]^{\perp})-\dim S_n\\
 &=\dim T_n(\R^d)-\dim[V_{n-1},\R^d]-\dim T_n(\R^d)+\dim S_n^{\perp}\\
 &=\dim S_n^{\perp}-\dim [V_{n-1},\R^d]\\
 &=\dim V_n-\dim[V_{n-1},\R^d] \\
 &=\dim (V_n / \dim[V_{n-1},\R^d] ).
\end{align*}

\textbf{What we need to solve Conjecture \ref{conj:loopinvareSplusareaconj} is the dimension of the intersection of the conjugation invariants with $S$, 
i.e. the dimension of the conjugation invariants which are zero on all loops.}

By
\Cref{prop:conj_invariants}
\begin{align*}
    \conjInvariants = [T(\R^d), \R^d]^\perp.
\end{align*}

\begin{definition}
 We define the space of \DEF{letter-reduced conjugation invariants} as $[T(\R^d),\R^d]^{\perp}/S$. 
\end{definition}
Note that $S$ is \emph{not} a subspace of $[T(\R^d),\R^d]^{\perp}$,
the quotient is to be read as
\begin{align*}
[T(\R^d),\R^d]^{\perp}/S
:=([T(\R^d),\R^d]^{\perp}+S)/S
\cong [T(\R^d),\R^d]^{\perp} / ([T(\R^d),\R^d]^{\perp} \cap S).
\end{align*}

Since $T(\R^d) \supset V$, we have $[T(\R^d),\R^d]\supset [V,\R^d]$, and thus
\begin{align*}
   [T(\R^d),\R^d]^\perp\subset [V,\R^d]^\perp.
\end{align*}

Since $S\subset [V,\R^d]^\perp$,
we have $[T(\R^d),\R^d]^\perp\subset[V,\R^d]^\perp$,
and so finally
\begin{equation*}
  [T(\R^d),\R^d]^{\perp}/S
  =([T(\R^d),\R^d]^{\perp}+S)/S
  \subset [V,\R^d]^\perp/S.
\end{equation*}

Then
\begin{align*}
\dim [T_{n-1}(\R^d),\R^d]^{\perp}/S_n&=\dim ([T_{n-1}(\R^d),\R^d]^{\perp}+S_n)/S_n\\
&=\dim([T_{n-1}(\R^d),\R^d]^{\perp}+S_n)-\dim S_n\\
&=\dim S_n^{\perp}-\dim([T_{n-1}(\R^d),\R^d]^{\perp}+S_n)^{\perp}\\
&=\dim V_n - \dim ([T_{n-1}(\R^d),\R^d]\cap V_n)\\
&=\dim (V_n/[T_{n-1}(\R^d),\R^d])
\end{align*}

Based on Tables \ref{tab:2dloopinv} and \ref{tab:3dloopinv}
we have the following stronger conjecture for two and three dimensions.
\begin{conjecture}
  \label{conj:loopinvareSplusareaconj2d3d}
For $d=2,3$, 
\begin{align*}
    [T(\R^d),\R^d]\cap V=[V,\R^d]\oplus\{\word{ij}-\word{ji}|\word{i},\word{j}=\word{1},\dots,\word{d}\},
\end{align*}
   and consequently $\loopInvariants=S+\conjInvariants+\{\word{ij}-\word{ji}|\word{i},\word{j}=\word{1},\dots,\word{d}\}$.
\end{conjecture}

Table \ref{tab:4dloopinv} shows that this does not hold for $d\geq 4$.

\begin{table}\label{tab:2dloopinv}
\centering
\begin{tabular}{c||cccccc}
\hline
Level&
\shortstack{conjugation\\invariants\\$[T_{n-1},\R^2]^\bot$\\\oeis{A000031}}&\shortstack{minimal $\shuffle$-gens\\of conj. inv.} & $V_n$ & $[V_{n-1},\mathbb R^2]$ & \shortstack{$\im\rcl\circ\rot$\\(letter-reduced\\conj. inv.)} & \shortstack{letter-reduced\\loop inv.}\\
\hline
1&2&2 & 0 & 0 & 0 & 0\\
2&3 & 0 & 1 & 0 & 0 & 1\\
3&4 & 0 & 2 & 2 & 0 & 0\\
4&6 & 1 & 4 & 3 & 1 & 1\\
5&8 &0 & 8 & 8 & 0 & 0\\
6&14&4 & 16 & 12 & 4 & 4\\
7&20 & 0 & 32 & 32 &0 & 0\\
8&36 & 9 & 64 & 54 &10 & 10\\
9&60 & 8 & 128 & 120 &8 & 8\\
10 & 108 & 20 & 256 & 232 & 24 & 24\\
11 & 188 & 32 & 512 & 480 & 32 & 32\\
12 & 352 & 68 & 1024 & 940 &  &\\
13 & 632 &   & 2048 & 1932 &  &\\
\hline
\end{tabular}
\caption{Sizes of various sets for two-dimensional paths, ignoring level 0. $\rcl$ is the right closure operator.
The conjugation invariants are not free
(otherwise the number of minimial generators
on level $12$ would be equal to $64$).
}
\end{table}

\begin{table}\label{tab:3dloopinv}
\centering
\begin{tabular}{c||cccccc}
\hline
Level&
\shortstack{conjugation\\invariants\\$[T_{n-1},\R^3]^\bot$\\\oeis{A001867}}&\shortstack{minimal $\shuffle$-gens\\of conj. inv.} & \shortstack{$V_n$\\\oeis{A118264}}& $[V_{n-1},\mathbb R^3]$ & \shortstack{$\im\rcl\circ\rot$\\(letter-reduced\\conj. inv.)} & \shortstack{letter-reduced\\loop inv.}\\
\hline
1&3&3 & 0 & 0 & 0 & 0\\
2&6 & 0 & 3 & 0 & 0 & 3\\
3&11 & 1 & 8 & 8 & 0 & 0\\
4&24 & 6 & 24 & 18 & 6 & 6\\
5&51 & 6 & 72 & 66 & 6 & 6\\
6&130 & 38 & 216 & 178 & 38 & 38\\
7& 315 & 54 & 648 & 594 & 54 & 54\\
8& 834 & & 1944 & 1716 & 228 & 228\\
9& 2195 & & 5832 & 5324 & 508 & 508\\
10& 5934 & & 17496 & 15960 &\\
\hline
\end{tabular}
\caption{Sizes of various sets for three-dimensional paths, ignoring level 0. $\rcl$ is the right closure operator.}
\end{table}

\begin{table}\label{tab:4dloopinv}
\centering
\begin{tabular}{c||cccccc}
\hline
Level&
\shortstack{conjugation\\invariants\\$[T_{n-1},\R^4]^\bot$\\\oeis{A001868}}&\shortstack{minimal $\shuffle$-gens\\of conj. inv.} & \shortstack{$V_n$\\\oeis{A118265}} & $[V_{n-1},\mathbb R^4]$ & \shortstack{$\im\rcl\circ\rot$\\(letter-reduced\\conj. inv.)} & \shortstack{letter-reduced\\loop inv.}\\
\hline
1&\hphantom{0}4& 4& 0 & 0 & 0 & 0\\
2&\hphantom{0}10& 0& 6 & 0 & 0 & 6\\
3&\hphantom{0}24& 4& 20 & 20 & 0 & 0\\
4&\hphantom{0}70 &  20& 81 & 60 & 20 & 21\\
5&\hphantom{0}208& 36& 324 & 288 & 36 & 36\\
6&\hphantom{0}700 & & 1296 & 1094 & 202 & 202\\
7& 2344 & & 5184 & 4648 & 536 & 536\\
8& 8230 & & 20736 & 18444 & 2292 & 2292\\
9& 29144 & & 82944 &  & & \\
10& 104968 & & 331776 & & & \\
\hline
\end{tabular}
\caption{Sizes of various sets for four-dimensional paths, ignoring level 0. $\rcl$ is the right closure operator.
The algebra is not free (otherwise the number of minimial generators
on level $4$ would be equal to $19$).
}
\end{table}

\begin{table}\
\centering
\begin{tabular}{c||cccccc}
\hline
Level&
\shortstack{conjugation\\invariants\\$[T_{n-1},\R^5]^\bot$\\\oeis{A001869}}&\shortstack{minimal $\shuffle$-gens\\of conj. inv.} & \shortstack{$V_n$\\\oeis{A118266}} & $[V_{n-1},\mathbb R^5]$ & \shortstack{$\im\rcl\circ\rot$\\(letter-reduced\\conj. inv.)} & \shortstack{letter-reduced\\loop inv.}\\
\hline
1&5& 5 & 0 & 0 & 0 & 0\\
2&15& 0 & 10 & 0 & 0 & 10\\
3&45& 10 & 40 & 40 & 0 & 0\\
4&165 & 50  & 205 & 150 & 50 & 55\\
5&629& 127 & 1024 & 898 & 126 & 126\\
6&2635 & & 5120 & 4360 & 760 & 760\\
7& 11165 & & 25600 & 22760 & 2840 & 2840\\
8& 48915 & & 128000 & 114070 & 13930 & 13930\\
9& 217045 & & 640000 &  &\\
10& 976887 & & 3200000 & &\\
\hline
\end{tabular}
\caption{Sizes of various sets for five-dimensional paths, ignoring level 0. $\rcl$ is the right closure operator. 
(The number of minimal generators of a free algebra would be $45$
on level $4$.)}
\end{table}

\begin{table}
\centering
\begin{tabular}{cllcccc}
\hline
Level&
\shortstack{conjugation\\invariants\\$[T_{n-1},\R^6]^\bot$\\\href{https://oeis.org/A054625}{\nolinkurl{A054625}}}&\shortstack{minimal $\shuffle$-gens\\of conj. inv.} & $V_n$ & $[V_{n-1},\mathbb R^6]$ & \shortstack{$\im\rcl\circ\rot$\\(letter-reduced\\conj. inv.)} & \shortstack{letter-reduced\\loop inv.}\\
\hline
1&6& 6& 0 & 0 & 0 & 0\\
2&21& 0& 15 & 0 & 0 & 15\\
3&76& 20& 70 & 70 & 0 & 0\\
4&336 & 105  & 435 & 315 & 105 & 120\\
5&1560& & 2604 & 2268 & 336 & 336\\
6&7826 & & 15625 & 13356 & 2268 & 2269\\
7& 39996 & & 93750 & 83010 & 10740 & 10740\\
8& 210126 & & 562500 &  & \\
9& 1119796 & &  &  &\\
10& 6047412 & &  & &\\
\hline
\end{tabular}
\caption{Sizes of various sets for six-dimensional paths, ignoring level 0. $\rcl$ is the right closure operator. First column is \url{https://oeis.org/A054625}.
(The number of minimal generators of a free algebra would be $90$
on level $4$.)}
\end{table}

\begin{conjecture}
  \label{conj:shuffle_of_letters}
In $d=2$, no homogeneous generator of the conjugation invariants for level $n\geq 2$ is a shuffle with letters, i.e.\ for $n\geq 2$, generators of the conjugations invariants are generators of the letter-reduced conjugation invariants. Or in other words, any conjugation invariant that is a shuffle with letters is a shuffle of conjugation invariants with letters.
\end{conjecture}

In $d=3$, this is not true because of signed volume.
Indeed,
$\rot\word{123}-\rot\word{132}=\word{1}\shuffle(\word{23}-\word{32})-\word{2}\shuffle(\word{13}-\word{31})+\word{3}\shuffle(\word{12}-\word{21})$.
However, we conjecture that the signed volume is the only exception in $d=3$.

\section{Closure invariants}
\label{sec:closure_invariants}
Right-closing a path means completing it into a closed path by adding a straight line from its end point to its start point.
We consider signature elements which do not change when you do this.

\begin{definition}
 An element $\tau\in T(\R^d)$ is called a \DEF{right-closure invariant} if
 \begin{equation*}
  \langle\sig(X),\tau\rangle=\langle\sig(X\sqcup R_X),\tau\rangle
 \end{equation*}
 for all paths $X$ and $R_X$ the linear path which has the
 same increment as the reversed path $\backwards{X}$.

 It is a \DEF{left-closure invariant}
 if
 \begin{equation*}
  \langle\sig(X),\tau\rangle=\langle\sig(R_X \sqcup X ),\tau\rangle
 \end{equation*}
 for all paths $X$.
 
\end{definition}

\begin{proposition}
    \label{prop:equivalent_closure}
    
    For an element $\EL \in T(\R^d)$ following are equivalent:
    \begin{enumerate}
        \item $\phi$ is right-closure invariant.
        \item
        For all $g$ grouplike
        with zero increment and any $z \in \R^d$,
        \begin{align}\label{eq:rightclosureinv}
            \langle g \exp(z), \EL \rangle
            =
            \langle g, \EL \rangle.
        \end{align}
        
        \item
    For all $v\in V$, all $z\in \R^d$ and for all $k\geq 1$
    \begin{align*}
       \langle v z^k , \EL \rangle = 0.
    \end{align*}
\end{enumerate}
\end{proposition}
\begin{proof}

   Let $\EL$ be a right closure invariant.
   Then, for any loop $O$ and any linear path $L$,
   \begin{equation*}
    \langle\sig(O\sqcup L),\EL\rangle=\langle\sig(O\sqcup L\sqcup \backwards{L}),\EL\rangle=\langle\sig(O),\EL\rangle
    \end{equation*}
    as $R_{O\sqcup L}=\backwards{L}$.
    Thus $\langle\sig(O)\exp(z),\EL\rangle=\langle\sig(O),\EL\rangle$ for arbitrary $z\in\R^d$,
    and due to Chen-Chow,
    this generalizes from $\sig(O)$ to any grouplike $g$ with zero increment.

   Let $\EL$ be such that it satisfies equation \eqref{eq:rightclosureinv},
   and $X$ be an arbitrary path.
   Then there is $a\in\R^d$ such that $\sig(R_X)=\exp(a)$,
   and so
   \begin{equation*}
    \langle\sig(X),\EL\rangle
    =\langle\sig(X\sqcup R_X\sqcup\backwards{R_X}),\EL\rangle
    =\langle\sig(X\sqcup R_X)\exp(-a),\EL\rangle
    =\langle\sig(X\sqcup R_X),\EL\rangle,
   \end{equation*}
   as $\sig(X\sqcup R_X)$ is grouplike with zero increment.

   \bigskip
    
Let $g$ be grouplike with zero increment, $z \in T(\R^d)$.
Now
$t\mapsto \langle g\exp(tz), \EL\rangle$ is a polynomial in $t$, and thus constant if and only if all derivatives in $t=0$ are zero.
Then, using 1. and 2.,
$\EL$ is a right-closure invariant if and only if
\begin{align*}
       \langle g z^{\itensor n} , \EL \rangle = 0,
\end{align*}
for all grouplike $g$ with zero increment, for all $n\geq 1$
and all linear combinations of letters $z$.
Using the linearization of Lemma \ref{lem:VnSn},
the claim follows.
\end{proof}

\begin{proposition}
    $\EL$ is both a left- and a right-closure invariant if and only if it is both a right-closure invariant and a loop invariant.
    We call these the \DEF{loop-and-closure invariants}.
\end{proposition}
\begin{proof}
    Let $u$ be both a left- and right-closure invariant.
    Then, for any loop $O$ and any linear $L$,
    \begin{equation}
     \langle\sig(O),u\rangle
     =\langle\sig(O\sqcup L),u\rangle
     =\langle\sig(\backwards{L}\sqcup O\sqcup L),u\rangle,
    \end{equation}
    where in the first equality we have used right-closure invariance,
    and left-closure invariance in the second.
    Iteratively, we see that $\langle\sig(\cdot),u\rangle$ is invariant under conjugation of loops by any piecewise linear paths,
    and thus, using Chen-Chow, we see that $u$ is a loop invariant.

    Let now $v$ be a right-closure invariant and a loop invariant.
    Then, for any loop $O$ and linear $L$,
    \begin{equation*}
     \langle\sig(O),v\rangle
     =\langle\sig(\backwards{L}\sqcup O\sqcup L),v\rangle
     =\langle\sig(\backwards{L}\sqcup O),v\rangle,
    \end{equation*}
    where we used loop invariance in the first equality,
    and right-closure invariance in the second equality.
    Since $L$ was arbitrary linear, $\backwards{L}$ is an arbitrary linear path,
    and so $v$ is a left-closure invariant.
\end{proof}

\begin{proposition}
  In $d=2$, the intersection of right-closure invariants with conjugation invariants is $\{0\}$.
\end{proposition}
\begin{proof}
  Let $u\in\im\rcl\circ\rot\cap\im\rot$.
  Let $O$ be any loop,
  and $O_1,O_2$ be any two paths such that $O=O_1\sqcup O_2$.
  Then,
  for any linear $L_1,L_2$,
  \begin{align*}
      \langle\sig(O),u\rangle
      &=\langle\sig(O_1\sqcup O_2),u\rangle
      =\langle\sig(O_2\sqcup O_1),u\rangle
      =\langle\sig(O_2\sqcup O_1\sqcup L_1),u\rangle\\
      &=\langle\sig(\backwards{L_2}\sqcup O_2\sqcup O_1\sqcup L_1\sqcup L_2),u\rangle
      =\langle\sig(\backwards{L_2}\sqcup O_2\sqcup O_1\sqcup L_1\sqcup L_2\sqcup\backwards{L_1}),u\rangle\\
      &=\langle\sig(O_1\sqcup L_1\sqcup L_2\sqcup \backwards{L_1}\sqcup\backwards{L_2}\sqcup O_2),u\rangle
  \end{align*}
 So $\langle\sig(\cdot),u\rangle$ is invariant under insertion of parallelograms $L_1\sqcup L_2\sqcup \backwards{L_1}\backwards{L_2}$ into loops.
 As up to tree-like equivalence,
 we may construct any two-dimensional piecewise linear loop by repeated insertion of parallelograms starting from the constant path (loop),
 $\langle\sig(\cdot),u\rangle$ is thus zero on any piecewise linear loop.
 Thus by closure invariance it is zero on any piecewise linear path,
 and so $u=0$ by Chen-Chow.
\end{proof}

Conjecturally, this holds for every $d$.
\begin{conjecture}
  \label{conj:closure_conj}
    For any $d$, the intersection of right-closure invariants with conjugation invariants is $\{0\}$.
\end{conjecture}

\begin{definition}
    The \DEF{right-closure operator} is defined as
    \begin{equation*}
        \rcl=\shuffle\circ (\operatorname{id}\otimes H)\circ \Delta_{\bullet}
    \end{equation*}
    where $H(\word{i}_1\dots\word{i}_n):=\frac{(-1)^n}{n!}\word{i}_1\shuffle\cdots\shuffle\word{i}_n$.
    The following lemma shows, that it
    is the algebraic realization
    of right-closing a path.
\end{definition}
\begin{lemma}\label{lem:rcl}
    \begin{equation*}\langle\sig(X),\rcl(w)\rangle
    =\langle\sig(X\sqcup R_X),w\rangle,\end{equation*} with $R_X$ as above.
\end{lemma}
\begin{proof}

First,
\begin{align*}
    \langle \sig(R_X), v_1 \dots v_n \rangle
    &=
    \langle \exp( - X_{0,1} ), v_1 \dots v_n \rangle \\
    &=
    \frac{(-1)^n}{n!}
    \langle X_{0,1}, v_1 \rangle \dots 
    \langle X_{0,1}, v_n \rangle \\
    &=
    \frac{(-1)^n}{n!}
    \langle \sig(X), v_1 \rangle \dots 
    \langle \sig(X), v_n \rangle \\
    &=
    \langle \sig(X), H(v) \rangle.
\end{align*}
Then
\begin{align*}
    \langle \sig(X \sqcup R_X), w \rangle
    &=
    \langle \sig(X) \sig(R_X), w \rangle \\
    &=
    \langle \sig(X) \otimes \sig(R_X), \Delta_\bullet w \rangle \\
    &=
    \langle \sig(X) \otimes \sig(X), (\id \otimes H) \circ \Delta_\bullet w \rangle \\
    &=
    \langle \Delta_\shuffle \sig(X), (\id \otimes H) \circ \Delta_\bullet w \rangle \\
    &=
    \langle \sig(X), \shuffle \circ (\id \otimes H) \circ \Delta_\bullet w \rangle.
\end{align*}
    
\end{proof}

\begin{corollary}
 $\rcl$ is idempotent (a projection),
 and $\EL$ is a right-closure invariant if and only if $\EL\in \im \rcl$.
\end{corollary}
\begin{proof}
 Immediate from Lemma \ref{lem:rcl}.
\end{proof}

\begin{lemma}
    \label{lem:ker_rcl_S}
    \begin{align*}
        \ker \rcl = S.
    \end{align*}
\end{lemma}
\begin{proof}

    $\EL \in S$ if and only if for all grouplike $g$ with zero increment
    \begin{align*}
        \langle g, \EL \rangle = 0. 
    \end{align*}
    Now
    \begin{align*}
        \langle g, \EL \rangle
        =
        \langle g, \rcl(\EL) \rangle.
    \end{align*}
    So that, if $\rcl(\EL) = 0$, we have $\EL \in S$.

    If $\EL \in S$, then for all grouplike $g$ with zero increment
    \begin{align*}
        \langle g, \EL \rangle = 0. 
    \end{align*}
    Then
    \begin{align*}
        0
        = \langle g, \rcl(\EL) \rangle
        = \langle g', \rcl(\EL) \rangle,
    \end{align*}
    for any $g'$ that closes to $g$.
    Since $g$ was arbitrary with zero increment,
    we get
    \begin{align*}
        0 = \langle g', \rcl(\EL) \rangle,
    \end{align*}
    for all grouplike elements $g'$,
    which shows $\rcl(\EL) = 0$.
\end{proof}

\newcommand\proj{\mathsf{proj}}
\begin{proposition}
 The dimension of the right closure invariants (equal to the dimension of the left closure invariants)
 is given by
 \begin{equation*}
  \dim\proj_n\im\rcl=\dim T_n(\R^d)-\dim S_n=\dim V_n,
 \end{equation*}
 and thus given by Corollary \ref{cor:dimVn}
\end{proposition}

\begin{proposition}
 The loop-and-closure invariants are isomorphic, as a graded algebra, to the letter-reduced loop invariants.
\end{proposition}
\begin{proof}
 The isomorphism is given by $[v]\mapsto \rcl(v)$, $v\in \loopInvariants$.
 This map is well-defined and injective as $\ker \rcl=S$, \Cref{lem:ker_rcl_S}.
 It is graded, as $\rcl$ is graded,
 it is an algebra homomorphism,
 as $\rcl$ is a shuffle homomorphism and $S$ is a shuffle ideal.
 Finally, it is surjective,
 as the loop-and-closure invariants are contained both in the image of $\rcl$
 and in the loop invariants,
 and $\rcl$ is a projection.
\end{proof}

This isomorphism is compatible with signatures of loops in the sense that
\begin{equation*}
 \langle\sig(O),[v]\rangle:=\langle\sig(O),v\rangle=\langle\sig(O),\rcl(v)\rangle
\end{equation*}
for any $v\in\loopInvariants$.

\begin{lemma}

    \begin{align*}
        \rcl \loopInvariants \subset \loopInvariants     
    \end{align*}
    
\end{lemma}
\begin{proof}
    Let $u\in\loopInvariants$.
    Then, for any loop $O$ and any path $X$,
    \begin{equation*}
     \langle\sig(O),\rcl(u)\rangle=\langle\sig(O),u\rangle=\langle\sig(\backwards{X}\sqcup O\sqcup X),u\rangle=\langle\sig(\backwards{X}\sqcup O\sqcup X),\rcl(u)\rangle,
    \end{equation*}
    and thus $\rcl(u)$ is a loop invariant.
\end{proof}

We conclude
\begin{equation*}
    \rcl\loopInvariants\cong[V,\R^d]^{\perp}/S
\end{equation*}
since the kernel of $\rcl$ is precisely $S$.

\begin{corollary}
    The intersection of right-closure invariants with $S$ is $\{0\}$.
\end{corollary}
\begin{proof}
    $T(\R^d)=S\oplus \im\rcl$, since $\rcl$ is a projection.
\end{proof}

\begin{proposition}
  $\loopInvariants=\ker (\rcl-\lcl)$
\end{proposition}
\begin{proof}
 If $\EL\in\loopInvariants$, then
 \begin{equation*}
  \langle\sig(X),\rcl(\EL)\rangle=\langle\sig(X\sqcup R_X),\EL\rangle=\langle \sig(R_X\sqcup X),\EL\rangle=\langle\sig(X),\lcl(\EL)\rangle
 \end{equation*}
 for arbitrary $X$,
 so $\rcl(\EL)=\lcl(\EL)$.

 If $\rcl(y)=\lcl(y)$,
 then
 \begin{equation*}
 \langle \sig(R_X\sqcup X),y\rangle=\langle\sig(X),\lcl(y)\rangle=\langle\sig(X),\rcl(y)\rangle=\langle\sig(X\sqcup R_X),y\rangle
 \end{equation*}
 so iteratively, $\langle\sig(\cdot),y\rangle$ is invariant under conjugation with piecewise linear paths,
 and so by Chen-Chow $y$ is a loop invariant.
\end{proof}

\begin{proposition}
\label{prop:rclrotarea}
We have
\begin{equation*}
    (\word{12}-\word{21})^{\shuffle 2}=2\cdot\rcl\circ\rot(\word{1212}).
\end{equation*}

More generally
\begin{equation*}
    (\word{12}-\word{21})\shuffle(\word{13}-\word{31})=2\cdot\rcl\circ\rot(\word{1213})
\end{equation*}
and
\begin{equation*}
    (\word{12}-\word{21})^{\shuffle 3}=4\cdot\rcl\circ\rot(\word{121212})+16\cdot\rcl\circ\rot(\word{121122}),
\end{equation*}
or more generally
\begin{equation*}
    (\word{12}-\word{21})^{\shuffle 2}\shuffle(\word{13}-\word{31})=4\cdot\rcl\circ\rot(\word{121213})+8\cdot\rcl\circ\rot(\word{121123})+8\cdot\rcl\circ\rot(\word{212113})
\end{equation*}
and
\begin{align*}
    &(\word{12}-\word{21})\shuffle(\word{13}-\word{31})\shuffle(\word{23}-\word{32})\\
    &\qquad=-8\cdot\rcl\circ\rot\word{121323}-16\cdot\rcl\circ\rot\word{212133}-16\cdot\rcl\circ\rot\word{122133}.
\end{align*}
Thus $(\word{ij}-\word{ji})^{\shuffle n}$ is contained in $\im\rcl\circ\rot$ for $n\geq 2$,
and more generally $(\word{ij}-\word{ji})^{\shuffle n}\shuffle(\word{ik}-\word{ki})^{\shuffle m}\shuffle(\word{jk}-\word{kj})^{\shuffle l}$ is contained in $\im\rcl\circ\rot$ for $n+m+l\geq 2$.

However, $(12-21)\shuffle(34-43)$ is \emph{not} contained in $\im\rcl\circ\rot$ for $d\geq 4$.

\end{proposition}

Note that we have $u=v+\text{shuffle of letters}$ if and only if $\rcl(u)=\rcl(v)$,
so shuffles of areas can be considered letter-reduced conjugation invariants
if and only if they are in $\im\rcl\circ\rot$.

\begin{conjecture}
  \label{conj:rot}
$(\word{i}_1\word{i}_2-\word{i}_2\word{i}_1)\shuffle\dots\shuffle(\word{i}_{2k-1}\word{i}_{2k}-\word{i}_{2k}\word{i}_{2k-1})$ is \emph{not} an element of $\im\rcl\circ\rot$ if all $\word{i}_j$ are distinct,
but \emph{is} contained in $\im\rcl\circ\rot$ if the letters are non-distinct.

Furthermore, $(\word{ij}-\word{ji})\shuffle u\in\rcl\circ\rot$ for all $u\in\rcl\circ\rot$.
\end{conjecture}

\begin{proposition}
If $d\ge2$, there are infinitely many algebraically (shuffle) independent loop-and-closure invariants in $T(\R^d)$.
\end{proposition}
The proof is very similar to that of Proposition \ref{prop:infmanyconjinv}.
\begin{proof}
Let n be arbitrary. We look at the axis parallel path with $8n$ steps given by $B_1\sqcup\dots B_{2n}$
where each $B_i$ is a rectangle, 
given by $B_i=(a_i\cdot e_1)\sqcup e_2 \sqcup (-a_i\cdot e_1)\sqcup (-e_2)$.

Under the signature, the loop-and-closure invariants
\begin{equation*}
 \rcl\circ\rot(\word{1}^{\bullet 2m+1}(\word{2}\word{1}\word{2}\word{1})^{\bullet 2n-1}\word{2})
\end{equation*}
evaluate to
\begin{equation*}
 \frac{(-1)^n}{n!}a_1^2\cdots a_{2n}^2\cdot\sum_{i=1}^{2n} 2a_i^{2m}
\end{equation*}
By algebraic independence of $\big(\sum_{i=1}^{2n} a_i^{2m})_{m=1}^n$ we obtain that there are at least $n$ shuffle-
independent loop-and-closure invariants.

Since $n$ was arbitrary, by the basis exchange property for algebraic independence, we thus
conclude that there exists a set of infinitely many algebraically independent loop-and-closure
invariants.
\end{proof}

\section{Outlook}

In addition to the Conjectures
\ref{conj:loopinvareSplusareaconj},
\ref{conj:loopinvareSplusareaconj2d3d},
\ref{conj:shuffle_of_letters},
\ref{conj:closure_conj}
and \ref{conj:rot},
we are interested in the following questions:
\begin{itemize}

  \item
  As is known in the literature
  on higher parallel transport,
  and elaborated in a notation close to ours
  in
  \cite[Section 3]{chevyrev2024multiplicative},
  the iterated-integrals signature of a closed loop
  can be found in a related \emph{surface signature}.
  Does this viewpoint help to understand the loop invariants?

  \item What about expected signatures \cite{chevyrevlyons}? 
  Is there a nice algebraic description of expected signature varieties \cite[Section~7]{AFS18} and in particular expected loop invariants 
  for for example Brownian loops (see e.g.\ \cite[Section~3]{revuzyor} 
  and \cite{AnkerBougerolJeulin}\footnote{Though we only refer to the $\mathbb{R}^d$ valued finite time $T$ case for now}), 
  i.e.\ Brownian motion $B:\Omega\times [0,T]\to\mathbb{R}^d$ with drift $\mu$ and covariance matrix $\Sigma$, conditioned on the event that $B(0)=B(T)$?

  \item How do conjugation invariants, loop invariants and closure invariants relate for the iterated sums signature \cite{diehl2020time}? In this setting, the signature characterizes the time series not up to tree-like equivalence, but up to time-warping. 
  So the situation is certainly different.
  This is interesting future work, where we will build upon \cite{adin2021cyclic}.

  \item The maybe most important
  open question is the following:
  If we are given the value of all loop invariants for a loop $O$, then to what extent does this characterize $O$?
  In other words, what is the equivalence relation given by \emph{all the loop invariants of the loops $O$ and $O'$ agree?}
  This question is also posed in the Outlook of \cite{lotterpreiss24}, in their setting of cyclic polytopes / positive matrices.
  It is an analogous question to \cite[Conjecture~7.2]{bib:DR2018}, which will be answered in \cite{diehllyonsnipreiss}.
  \item We ask for a geometric interpretation of loop invariants beyond signed area and even dimensional signed volumes (cf.\ \cite{bib:DR2018}, \cite{amendola2023convex} for signed volumes).

\end{itemize}

\bibliographystyle{alpha} 
\bibliography{loopinvariants}

\end{document}